\documentclass[12pt,twoside,leqno]{amsart}
\usepackage{amsmath,amssymb,amsfonts,amsthm,amsopn}
\usepackage{graphics}
\def\ni{\noindent}
\def\simp{Sp(d,\bR)}
\def\Wpp{W_{\psi}}
\def\Wop{W_{\mu_e(h)\psi}(x,\xi)}

\def\Fur{\mathcal{F}}

\def\la{\langle}
\def\ra{\rangle}

\def\phas{(x,\xi )}
\def\sch{\mathcal{S}}
\def\f{\psi}

\def\b{\beta}
\def\intrd{\int_{\rd}}

\def\cF{\mathcal{F}}              
\def\cS{\mathcal{S}}
\def\rd{\bR^d}

\def\rdd{{\bR^{2d}}}

\def\lrd{L^2(\rd)}

\def\bR{{\mathbb R}}
\def\Ren{\mathbb{R}^d}

\def\R{{\mathbb R}}
\def\C{{\mathbb C}}
\def\H{{\mathbb H}}

\newcommand\mc[1]{{\mathcal{#1}}}

\def\Spnr{Sp(d,\R)}
\def\Sptwor{Sp(2,\R)}
\def\Gltwonr{GL(2d,\R)}

\def\Aut{\mathop{\rm Aut}}

\def\t{\!\;^t\!}   

\def\cC{{\mathcal C}}

\def\cE{{\mathcal E}}
\def\cF{{\mathcal F}}

\def\cS{{\mathcal S}}

\def\cU{{\mathcal U}}

\def\squareforqed{\hbox{\rlap{$\sqcap$}$\sqcup$}}

\def\qed{\ifmmode\squareforqed\else{\unskip\nobreak\hfil
\penalty50\hskip1em\null\nobreak\hfil\squareforqed
\parfillskip=0pt\finalhyphendemerits=0\endgraf}\fi}

\newtheorem{theorem}{\sc Theorem}
\newtheorem{cor}[theorem]{\sc Corollary}

\newtheorem{lemma}[theorem]{\sc Lemma}
\newtheorem{prop}[theorem]{\sc Proposition}
\newtheorem{proposition}[theorem]{\sc Proposition}
\newtheorem{defn}[theorem]{\sc Definition}

\newtheorem{ex}[theorem]{\sc Example}

\begin{document}

\title{Triangular Subgroups of $\Spnr$ and  Reproducing Formulae}
\author{ E.~Cordero}
\address{Dipartimento di Matematica\\
Universit\`a di  Torino\\
Italy}
\email{elena.cordero@unito.it}
\author{A.~Tabacco}
\address{Dipartimento di Scienze Matematiche\\
Politecnico di  Torino\\Italy}
\email{anita.tabacco@polito.it}

\subjclass[2010]{42C15, 17B45} \keywords{Reproducing formulae, Metaplectic representation, Wigner distribution,
 Semidirect product}
\date{June 11, 2012}
\begin{abstract}
We consider the (extended) metaplectic representation of the
semidirect product $\mathcal{G}=\H^d\rtimes Sp(d,\R)$ between the Heisenberg
group  and the symplectic group. Subgroups $H=\Sigma \rtimes D$, with $\Sigma$ being a $d\times d$ symmetric matrix and $D$ a closed
subgroup of $GL(d,\R)$, are our main concern. We shall give a general setting
for the reproducibility of such groups which include and assemble the ones for the single examples treated in  \cite{AEFK1}. As a byproduct,  the extended metaplectic representation restricted to some classes of such subgroups  is either the Schr\" odinger representation of $\rdd$ or the wavelet representation of
$\rd\rtimes D$, with $D$ closed subgroup of $GL(d,\R)$. Finally, we shall provide new examples of reproducing groups
of the type $H=\Sigma\rtimes D$, in dimension $d=2$.
\end{abstract}

\keywords{Reproducing formulae, Metaplectic representation, Wigner distribution, Semidirect product.} \maketitle



\pagestyle{myheadings} \markboth{\sc the equivalence}{\sc
e.cordero et al.}



\section{Introduction}

Reproducing formulae  appear
almost everywhere in the literature, from  coherent states in
physics to group representations, Gabor analysis,  wavelet analysis and its many generalizations. This theory has a wealth of applications in engineering, physics and numerical analysis (see, e.g., \cite{ALI2000,BernierTaylor96,Candes2001,DIX,FHN08,book,WW01} and references therein).
It is remarkable to observe that most existing reproducing formulae for functions $f\in\lrd$ can be formulated in one and the same general form, that is, an integral formula of the type

\begin{equation}\label{rep}
f=\int_{H}\la f,\mu_e(h)\phi\ra\mu_e(h)\phi\;dh,\quad \mbox{for\, all}\, f\in\lrd,
\end{equation}
in the following sense. First of all, the domain of integration $H$ is a (connected, closed, Lie)
subgroup of the semidirect product $\mathcal{G}=\H^d\rtimes Sp(d,\R)$ between the Heisenberg
group $\H^d$ and the symplectic group $Sp(d,\R)$, $dh$ is a left Haar measure of $H$, and
$\mu_e$ is the extended metaplectic representation of  $\mathcal{G}$, to be defined below in detail.
In the current literature, the function $\phi\in\lrd$ is usually referred as wavelet or admissible vector.
As we shall show below, many standard reproducing formulae, such as those arising in Gabor analysis and wavelet theory,
are of the above type: for some groups $H$, but not for all, one finds  $\phi\in\lrd$  in such a way that the above formula holds weakly for every $f\in\lrd$. Thus, some groups $H$ give rise to interesting analysis and some do not, like, for example, the full factors themselves, i.e. $\H^d\subset \mathcal{G}$ or $Sp(d,\R)\subset \mathcal{G}$. The question whether
a given subgroup $H\subset \mathcal{G}$ leads to a reproducing formula or not is both relevant and difficult, and is the main theme
of our recent investigations \cite{AEFK1,AEFK2}, together with the explicit description of the admissible vectors.
The ongoing formulation of things may be simplified a bit because the reproducing formula ~\eqref{rep}
is equivalent to
\begin{equation}\label{rep2}
\|f\|_2^2=\int_{H}|\la f,\mu_e(h)\phi\ra|^2\;dh,\quad {\rm
for\,all}\, f\in\lrd,
\end{equation}
which is manifestly insensitive to phase multiplicative factors, that is, invariant under transformations of the type
$\phi\mapsto e^{i \alpha}\phi$. This allows a technical reduction, that is, one can factor out the center of the
Heisenberg group, whose action through $\mu_e$ is through phase factors, and one may safely pass from the whole group $\mathcal{G}$ to the somewhat simpler group $G=\rdd\rtimes Sp(d,\R)$. We formalize this discussion in the following:

\begin{defn}\label{RS} A connected Lie subgroup $H$ of
$G=\rdd\rtimes Sp(d,\R)$ is a {\it reproducing group} for $\mu_e$
if there exists a function $\phi\in\lrd$ such that
the identity \eqref{rep}
holds weakly for all $f\in\lrd$. Any such $\phi\in\lrd$ is called a reproducing function.
\end{defn}
Notice that we do require formula~\eqref{rep} to hold for all
functions in $\lrd$ for the same window $\phi$, but we do not
require the restriction of $\mu_e$ to $H$ to be irreducible.

\par
In this paper we consider a class of triangular subgroups of the symplectic group, that we collectively denote as the class $\cE$. From the structural point of view, a group $H\in\cE$ is a semidirect product of the form $H=\Sigma \rtimes D$,
where $\Sigma$ is a $d$-dimensional subspace of $d\times d$ symmetric matrices and $D$ is a closed subgroup of $GL(d,\R)$ acting on $\Sigma$; hence $\Sigma$ is abelian and normal in $H$. Evidently, when seen within $G$, each of these groups is contained in the symplectic factor.
We shall show that many examples used in the applications fall in this class. This motivates our study of the class $\cE$.\par Our main result is provided by Theorem \ref{RepThm} below, which contains necessary and sufficient conditions for a wavelet  $\psi$ to be reproducing on a group $H$ in $\cE$.  This result is a far-reaching extension  of the reproducing conditions for the special cases treated in \cite{AEFK2}. We underline that this result is based on a deep study of the properties of a quadratic  mapping $\Phi$ on $\rd$, developed in Section \ref{42} below.
Finally, we shall provide new examples of reproducing groups
in the class $\cE$ in dimension $d=2$.
\par
The paper also contains some additional remarks concerning an alternative formulation of the concept of \emph{admissible group}, a formally stronger notion than the notion given in Definition \ref{RS}. We also clarify in what sense our setup includes Gabor and wavelet analyses.

\section{Preliminaries and notation}\label{prelim}
The symplectic group is defined by
$$
\Spnr=\left\{g\in\Gltwonr:\;^t\!gJg=J\right\},
$$
where
$$
J=\bmatrix 0&I_d\\-I_d&0\endbmatrix
$$
is the standard  symplectic form
\begin{equation}
\omega(x,y)=\;^t\!xJy, \qquad x,y\in\R^{2d}.
\label{symp}\end{equation}

The metaplectic representation $\mu$ of (the two-sheeted cover of)
the symplectic group arises as intertwining operator between the
standard Schr\"odinger representation $\rho$ of the Heisenberg
group $\H^d$ and the representation that is obtained from it by
composing $\rho$ with the action of $\Spnr$ by automorphisms on
$\H^d$ (see, e.g., \cite{folland89}).  We briefly review its
construction.
\par
The Heisenberg group $\H^d$ is the group obtained by defining on
$\R^{2d+1}$  the product$$
(z,t)\cdot(z',t')=(z+z',t+t'+\frac{1}{2}\omega(z,z')),
$$
where $\omega$ stands for the standard symplectic form in
$\R^{2d}$ given in \eqref{symp}. We denote the   translation  and
modulation operators on $\lrd$ by
$$
T_xf(t)=f(t-x)\quad{\rm and}\quad M_{\xi}f(t)= e^{2\pi i\la \xi,
t\ra }f(t).$$ The Schr\"odinger representation of the group $\H^d$
on $\lrd$ is then defined by
$$
\rho(x,\xi,t)f(y)=e^{2\pi it}e^{-\pi i \langle x,\xi\rangle}
e^{2\pi i \langle\xi, y\rangle}f(y-x) = e^{2\pi it}e^{\pi i\langle
x, \xi\rangle} T_x M_\xi f(y),
$$
where we write $z=(x,\xi)$ when we separate space components (that
is $x$) from frequency components (that is $\xi$) in a point $z$
in phase space $\R^{2d}$. The symplectic group acts on $\H^d$ via
automorphisms that leave the center
$\{(0,t):t\in\R\}\in\H^d\simeq\R$ of $\H^d$ pointwise fixed:
$$
A\cdot\left(z,t\right) =\left(Az,t\right).
$$
Therefore, for any fixed $A\in Sp(d,\R)$ there is a representation
$$
\rho_A:\H^d\to\cU(L^2(\R^d)), \qquad \left(z,t\right)\mapsto
\rho\left(A\cdot(z,t)\right)
$$
whose restriction to the center is a multiple of the identity. By
the Stone-von Neumann  theorem, $\rho_A$ is equivalent to $\rho$.
That is, there exists an intertwining unitary operator
$\mu(A)\in\cU(L^2(\R^d))$ such that
$\rho_A(z,t)=\mu(A)\circ\rho(z,t)\circ\mu(A)^{-1}$, for all
$(z,t)\in \H^d$. By Schur's lemma, $\mu$ is determined up to a
phase factor $e^{is}, s\in\R$. It turns out that the phase
ambiguity is really a sign, so that $\mu$ lifts to a
representation of the (double cover of the) symplectic group. It
is the famous metaplectic or Shale-Weil representation. \par The
representations $\rho$ and $\mu$ can be combined and give rise to
the  extended metaplectic representation of the group
$G=\H^d\rtimes Sp(d,\R)$, the semidirect product of $\H^d$ and
$Sp(d,\R)$. The group law on $G$ is
\begin{equation}\label{Glaw}
\left((z,t),A\right)\cdot\left((z',t'),A'\right)
=\left((z,t)\cdot(Az',t'),AA'\right)
\end{equation}
and the extended metaplectic   representation $\mu_e$ of $G$ is
\begin{equation}\label{defex}
\mu_e\left((z,t),A\right)=\rho(z,t)\circ\mu(A).
\end{equation}
\par
Observe that the reproducing formula~\eqref{rep} is
insensitive to phase factors:  if we replace
$\mu_e(h)\phi$ with $e^{is}\mu_e(h)\phi$ the formula is
unchanged, for any $s\in\R$. The role of the center of the
Heisenberg group is thus irrelevant, so that the ``true'' group
under consideration is $\rdd\rtimes\Spnr$, which we denote again
by $G$. Thus $G$ acts naturally by affine transformations on phase
space, namely
\begin{equation}
g\cdot\phas=\left((q,p),A\right)\cdot\phas=A\t\phas+\t(q,p).
\label{affaction3}
\end{equation}

For elements $\Spnr$ in special form, the metaplectic
representation can be computed explicitly in a simple way. For
$f\in L^2(\R^d)$, we have
\begin{align}
\mu\left(\bmatrix A&0\\ 0&\;^t\!A^{-1}\endbmatrix\right)f(x)
&=(\det A)^{-1/2}f(A^{-1}x)\label{diag}\\
\mu\left(\bmatrix I&0\\ C&I\endbmatrix\right)f(x)
&=\pm e^{-i\pi\langle Cx,x\rangle}f(x)\label{lower}\\
\mu\left(J\right)&=i^{d/2}\mc{F}^{-1}\label{iot},
\end{align}
where $\mc{F}$ denotes the Fourier transform
$$
\mc{F} f(\xi)=\int_{\R^d}f(x)e^{-2\pi i\langle x,\xi\rangle}\;dx,
\qquad f\in L^1(\R^d)\cap L^2(\R^d).
$$
In the above formula and elsewhere,  $\la x,\xi\ra$ denotes the
inner product of $x,\xi\in\Ren$.  Similarly, for $f,g\in\lrd$,
$\la f,g\ra$ will denote their inner product  in $L^2(\Ren)$.
Other notation is as follows. We put $\dot\R=\R\setminus\{0\}$,
$\R_{\pm}=(0,\pm\infty)$. For $1\leq p\leq\infty$,  $\|\cdot\|_p$
stands for the $L^p$-norm of measurable functions on $\rd$ with
respect to Lebesgue measure. The left Haar measure of a group $H$
will be written $dh$ and we always assume that the Haar measure of
a compact group  is normalized so that the total mass is one.
Let $\Omega$ be an open set of $\rd$. Then $\cC_0^\infty(\Omega)$ is the space of smooth functions with compact support contained in $\Omega$.
\section{Gabor and Wavelet analyses}
We show below that both Gabor and wavelet analyses can be viewed as particular cases of the general setup that we are considering. This fact is somehow known and is perhaps part of common knowledge, but on the one hand we could not locate it in the literature in a precise fashion and, on the other hand, we want to present some additional remarks that are of some independent interest. We thus start with a side observation.

\subsection{Weak  admissibility}  The notion of reproducing group admits an equivalent version that is obtained by weakening a property that has been investigated in \cite{AEFK1} and that is formulated by means of the    Wigner
distribution. The cross-Wigner
distribution $W_{f,g}$ of $f,g\in L^2(\Ren)$ is given by
\begin{equation}
\label{eq3232} W_{f,g}(x,\xi)=\int  e^{-2\pi i\la\xi,y\ra}
f\left(x+\frac{y}2\right)\overline{g\left(x-\frac{y}2\right)}
\,dy.
\end{equation}
The quadratic expression $W_f := W_{f,f}$ is  called the {\it
Wigner distribution} of $f$. We collect below some of its well-known properties (see e.g. \cite{folland89}).
\begin{prop}\label{Wigflavor} The Wigner distribution of $f\in\lrd$
satisfies:
\begin{itemize}
\item[(i)]  $W_f$ is uniformly continuous on $\rdd$, and
$\|W_f\|_{\infty}\leq 2^d \|f\|_2^2$. \item[(ii)]
$W_{f}$ is
real-valued.
 \item[(iii)]  Moyal's identity:
${\displaystyle \la W_{f},W_{g}\ra_{L^2(\rdd)}=\la
f,g\ra_{L^2(\rd)} \overline{\la f,g\ra}_{L^2(\rd)}}$. \item[(iv)]
If $f\in\sch(\R^d)$, then $W_f\in \sch(\R^{2d})$.
\item[(v)]  Marginal properties:
\begin{equation}\label{mmm}
\intrd W_f(x,\xi)\,d\xi=|f(x)|^2,
\qquad
\intrd W_f(x,\xi)\,dx=|\hat{f}(\xi)|^2
\end{equation}
for $\hat{f}\in L^1(\R^d)$, $f\in L^1(\R^d)$, respectively.
\item[(vi)] If both $f,\hat{f}$ are in $L^1(\R^d)$ (hence in $L^2(\R^d)$)  then
\begin{equation}
\int_{\rdd}\,W_f(w)\,dw=\|f\|_{L^2}^2.
\label{eqnorm}
\end{equation}
\end{itemize}
\end{prop}
In \cite{AEFK1} it is introduced the notion of \emph{admissible} group, one for which there exists $\phi\in\lrd$ such that \eqref{admiss} below holds. Together with
some additional integrability and boundedness properties of
$h\mapsto \Wpp(h^{-1}\cdot\phas)$, it implies  that a subgroup $H$ of
$G=\rdd\rtimes\simp$ is reproducing. For the reader's convenience, we recall the statement of  \cite[Thm.1]{AEFK1}, where the main  point is made.
\begin{theorem}\label{main}
Suppose that $\phi\in\lrd$ is such that the mapping
\begin{equation}\label{suff1}
h\mapsto \Wop=\Wpp(h^{-1}\cdot\phas)\end{equation} is in
$L^{1}(H)$ for a.e. $\phas\in\rdd$ and
\begin{equation}\label{suff2} \int_H |\Wpp(h^{-1}\cdot\phas)|\,dh\leq
M, \qquad \text{for a.e. }\phas\in\rdd.
\end{equation}
Then condition \eqref{rep} holds for all $f\in\lrd$ if and only if
the following admissibility condition is satisfied:
\begin{equation}\label{admiss}
\int_H\Wpp(h^{-1}\cdot\phas)\,dh=1,\quad{\mbox
for\,a.e.}\,\phas\in\rdd.
\end{equation}
\end{theorem}
Assumptions  \eqref{suff1} and \eqref{suff2} are sufficient but not necessary conditions for a subgroup to be  reproducing, as illustrated below (see the example in the next Section \ref{classeE}).

The notion of \emph{weak admissible} group is as follows.
Consider first the vector space
\begin{equation}\label{Vspan}V:={\rm span}\{W_f,\,\,f\in\cS(\rd)\}=\{ \sum_{k=1}^N c_k W_{f_k}, f_k\in\cS(\rd), c_k\in\C\}.\end{equation}
For $f,g\in\sch(\rd)$, a straightforward computation gives
\begin{equation}\label{pol} W_{f,g}=\frac12 \left[W_{f+g}+i W_{f+ig}-(1+i)(W_f+W_g))\right].
\end{equation}
Since  ${\rm
span}\,\{W_{f,g}\,\,f,g\in\sch(\rd)\}$ is dense in $\sch(\rdd)\,$ (see  \cite{AEFK1}), it follows from \eqref{pol} that also $V$ is dense in $\sch(\rdd)$.

Now, assume that the subgroup $H$ is reproducing and let $\phi\in\lrd$ be a reproducing function.
Define the conjugate-linear functional $\ell$ on $V$ by
\begin{equation}\label{funt}\ell(F)=\int_H \la W_{\mu_e(h)\phi}, F\ra \,dh,\quad \forall F\in V.\end{equation}
The functional $\ell$ is well-defined and continuous on $V$ with respect to the $L^1$ norm, as  shown presently.
Let  $F=\sum_{k=1}^N c_k W_{f_k}, f_k\in\cS(\rd),$ $c_k\in\C$, be an element of the space $V$. Using Moyal's identity (Proposition \ref{Wigflavor})  and the reproducing condition \eqref{rep2}, we have
\begin{align*}
|\ell(F)|&=|\int_H \la W_{\mu_e(h)\phi}, \sum_{k=1}^N c_k W_{f_k}\ra \,dh|=|\sum_{k=1}^N\bar{c}_k\int_H \la W_{\mu_e(h)\phi}, W_{f_k}\ra \,dh|\\
&=|\sum_{k=1}^N \bar{c}_k\int_{H}|\la f_k,\mu_e(h)\phi\ra|^2\,d h|=|\sum_{k=1}^N\bar{c}_k\|f_k\|^2_2|=|\sum_{k=1}^N\bar{c}_k \int_\rdd W_{f_k}\phas dx d\xi|\\
&=|\int_\rdd \sum_{k=1}^N\bar{c}_k W_{f_k} \phas dx d\xi |=|\la 1, \sum_{k=1}^N {c}_k W_{f_k}\ra| =|\int_\rdd \overline{F}\phas dx d\xi|\leq \|F\|_1.
\end{align*}
Moreover, from the previous computations it is clear that $\ell(F)= \la 1,F\ra$, for every $F\in V$.
Finally, since $V$ is dense in $L^1(\rdd)$, the functional $\ell$ can be  uniquely extended to a continuous functional $\tilde{\ell}$ on $L^1(\rdd)$, which   coincides with $1\in L^\infty(\rdd)$.

Observe that, in general, this does not imply that $\tilde{\ell}(F)=\int_H \la W_{\mu_e(h)\phi}, F\ra \,dh,\quad \forall F\in L^1(\rdd)$. Indeed, one should know that the mapping $F\to \la W_{\mu_e(h)\phi}, F\ra$ is in $L^1(H)$  in order for  the integral to make sense, and it should satisfy $\int_{H}\la W_{\mu_e(h)\phi}\, dh, F\ra=\la 1,F\ra,\,\,\forall F\in L^1(\rdd).$

Those observations yield to the following definition.
\begin{defn}
We say that the subgroup $H$ of $G$ is  weakly admissible if there exists a function $\phi\in\lrd$ such that the functional
\eqref{funt} is well-defined on the set $V$ defined in \eqref{Vspan} and verify
\begin{equation}\label{funt2} \ell(F)=\int_H \la W_{\mu_e(h)\phi}, F\ra \,dh=\la 1, F \ra,\quad \forall F\in V.\end{equation}
\end{defn}
We are now in a position to state and prove our first observation.

\begin{theorem}\label{eqrepadm} The following  are equivalent:\\
(i) $H$ is weakly admissible and \eqref{funt2} holds for $\phi\in\lrd$;
(ii) $H$ is reproducing and $\phi\in\lrd$ is a reproducing function.
\end{theorem}
\begin{proof}
Observe that the reproducibility condition \eqref{rep2} can be checked on the dense subspace   $\cS(\rd)\subset L^2(\rd)$.\\
$(i)\Rightarrow (ii)$.  Take $f\in\cS(\rd)$ and use the admissibility condition with $F=W_f$. Then,
Moyal's identity and $\la 1, W_f\ra=\| f\|_2^2$ immediately provide the desired result.\\
\noindent
$(ii)\Rightarrow (i)$. It follows by the previous observations.
\end{proof}

\textbf{Remark.} If the assumptions  \eqref{suff1} and \eqref{suff2} of Theorem \ref{main} hold, then the mapping $F\to \la W_{\mu_e(h)\phi}, F\ra$ is in $L^1(H)$ (one can apply Fubini Theorem and exchange the integrals) and, moreover,  $\int_{H}\la W_{\mu_e(h)\phi}, F\,\ra dh=\la \int_{H} W_{\mu_e(h)\phi}\, dh, F\ra=\la 1, F\ra, \quad \forall F\in L^1(\rdd).$ Hence the weak admissibility condition  generalizes  \eqref{admiss}.

\subsection{Gabor analysis} Gabor's reproducing formula is given by
$$f=\int_\rdd \la f, T_x M_\xi\psi \ra T_x M_\xi\psi \,dxd\xi,$$
which is (weakly) true for every $\psi\in\lrd$, with $\|\psi\|_2=1$.
We shall refer to this basic fact as to Gabor's theorem.  The extended metaplectic
representation $\mu_e$ restricted to the subgroup  $H\cong\R^{2d}$ consisting of the first factor in $G$ is, of course, the
Schr\" odinger representation $\rho$.
If we consider a function $\psi\in L^1(\rd)\cap \Fur L^1(\rd)$, so
that conditions \eqref{suff1} and \eqref{suff2} are satisfied, the
admissibility condition \eqref{admiss}  becomes $\int_\rdd
W_\psi(x-q,\xi-p)\,dpdq=1$, for a.e. $\phas\in\rdd,$ that is
\begin{equation}\label{GabWig}
\int_\rdd W_\psi(q,p)\,dpdq=1.
\end{equation}
We need the conditions \eqref{suff1} and \eqref{suff2} to have
$W_\psi\in L^1(\rdd)$ for granted. However, if we drop the
requirement  $W_\psi\in L^1(\rdd)$, the reproducing window $\psi$
can be rougher, as shown below.

\begin{proposition}\label{gb}
Let $\psi\in L^2(\rd)$ with $\|\psi\|_{L^2}=1$. \\
(i) If   $\psi\in
L^1(\rd)$, then $\displaystyle{\intrd\left(\intrd W_\psi(q,p) dq\right)
\,dp=1}$. \\
(ii) If $\hat{\psi}\in L^1(\rd)$, then $\displaystyle{\intrd\left(\intrd W_\psi(q,p) dp\right) \,dq=1}$.
\end{proposition}
\begin{proof}
If $\psi\in L^1(\rd)$ with $\|\psi\|_{2}=1$, then $\psi$ is
reproducing by the Gabor's theorem. By the
second marginal property in \eqref{mmm}, the map $q\rightarrow W_\psi(q,p)$
is integrable on $\rd$ for a.e. $p\in\rd$. Since
$|\hat\psi(p)|^2=\intrd W_\psi(q,p)\,dq$ and $\intrd
|\hat\psi(p)|^2\,dp=\|\hat\psi\|_2^2=\|\psi\|_2^2=1,$ the claim is
proved. If  $\hat{\psi}\in L^1(\rd)$,  we use the same arguments
with the first marginal property in \eqref{mmm}.
\end{proof}

\textbf{Remark.} The previous proposition indicates that \eqref{GabWig}
 may fail to be true even for a reproducing function $\psi\in\lrd$, but it is replaced by subtly weaker conditions for integrable reproducing functions (or for reproducing functions with integrable Fourier transform). Assumptions \eqref{suff1} and \eqref{suff2} are actually not necessary for
$\psi$ to be a reproducing function, as the simple example  below illustrates.

\begin{ex}\textnormal{
In dimension $d=1$, consider the box function
$\psi(x)=\chi_{[-1/2,1/2]}(x)$. Then $\psi\in L^1(\R)$ and
$\|\psi\|_{L^2}=1$, so that it is a reproducing function.  On the
other hand, conditions  \eqref{suff1} and \eqref{suff2} are not
fulfilled. This is seen  by  computing the Wigner
distribution $W_\psi$. Indeed,  using the definition
\eqref{eq3232}, we get
$$
W_{\psi}(x,\xi)=
\displaystyle\begin{cases}      \displaystyle \frac{\sin[2\pi(1+2 x)\xi]}{\pi \xi},&  x\in (-\frac12,0),\,\xi\not=0 \\
               2(1+2 x),          &   x\in (-\frac12,0),\,\xi=0 \\
              \displaystyle \frac{\sin[2\pi(1-2 x)\xi]}{\pi \xi},&  x\in [0,\frac12),\,\xi\not=0 \\
    2(1-2 x),          &   x\in  [0,\frac12),\,\xi=0 \\
 0, & \mbox{otherwise}.
\end{cases}
$$
Clearly, $W_\psi\notin L^1(\R^2)$, however, observe that
Proposition \ref{gb} is satisfied with the admissibility condition
$\int_{\R}(\int_{\R} W_\psi(x,\xi) d\xi)\,dx=1$.}
\end{ex}

The Gabor case is a particular example of a subgroup of the form
 $H=\rdd\rtimes K$, with $K$ subgroup of $Sp(d,\R)$
(here $K=\{I_{2d}\}$). The reproducibility of $H$ is equivalent to
asking the compactness of $K$. Indeed, if
 arbitrary translations are  allowed in the affine action of
$H$, then the symplectic factor must be compact, as shown below.

\begin{prop}\label{comp} If $H=\rdd\rtimes K$, with
$K\subset\simp$, then  $H$ is reproducing if and only if $K$ is
compact.
\end{prop}
\begin{proof}
The left Haar measure of $H$ is given by $dh=dx\,d\xi\,dk$, where
$dx\,d\xi$ is the Lebesgue measure on $\rdd$ and $dk$ is  the left
Haar measure of $K$. In the computations below, we take
$f\in\sch(\rd)$. We first write the right-hand side of
\eqref{rep2}, we then  apply Plancherel's theorem,  compute the
Fourier transform of the time-shift $T_x$, then use Parseval
Identity and, finally, the Fourier transform of the
frequency-shift $M_\xi$:
\begin{align*}
\int_K\int_\rdd|\la f,T_xM_\xi\mu(k)\phi\ra|^2 &\,dxd\xi dk
= \int_K\int_\rdd|\la {\hat f},\cF(T_xM_\xi\mu(k)\phi)\ra|^2\,dxd\xi dk\\
&= \int_K\int_\rdd\left|\int_{\rd}  {\hat f}(t)e^{2\pi i xt}
\;\overline{\cF(M_\xi\mu(k)\phi)}(t)dt\right|^2\,dxd\xi dk\\
&=\int_K\int_{\rd}\int_{\rd}\left|\cF^{-1}({\hat f}
\;\overline{\cF(M_\xi \mu(k)\phi}))(x)\right|^2 \,dxd\xi dk\\
&= \int_K\int_{\rd}\int_{\rd} |{\hat{f}}(\eta)|^2\,|
\cF(M_\xi \mu(k)\phi)(\eta)|^2 \,d\eta d\xi dk\\
&=\int_K\left(\int_{\rd}\left(\int_{\rd}
|\cF (\mu(k)\phi)(\eta-\xi)|^2 \,d\xi\right) |{\hat f}(\eta)|^2 d\eta\right) dk\\
&=\left(\int_{\rd} |{\hat
f}(\eta)|^2d\eta\right)\int_K\!\!\|\cF(\mu(k)\phi)\|^2_2 dk= \|
f\|_2^2 \int_K\!\!\|\mu(k)\phi\|^2_2 dk \\&=\| f\|_2^2\|\phi\|^2_2
\int_K dk=\| f\|_2^2\|\phi\|^2_2 mis(K).
\end{align*}
The  interchange in the order of  integration  is justified by
Fubini's theorem, since
$$ \int_{\rd}\int_{\rd} |{\hat {f}}(\eta)|^2\,|{\Fur({M_\xi
\mu(k)\phi}})(\eta)|^2 \,d\xi d\eta \leq\|\phi\|_2^2 \|f\|^2_2.
$$
Now, the Haar measure of the locally compact subgroup $K$ is
finite if and only if  $K$ is compact (see e.g., \cite{Fol2}).
This  concludes the proof.
\end{proof}

\subsection{Wavelet analysis.} We now examine the (generalized) continuous wavelet transform in higher dimensions, that is in $\rd$, with $d\geq 1$. It arises from restriction of the metaplectic representation to semidirect products of the form $\rd\rtimes D$, where $D$ is any closed subgroup of
$GL(d,\R)$. The product law is
\begin{equation}\label{Fpr} (q,a)\cdot (q',a')=( a q'+ q, a a'),\quad  q,q'\in\rd,\,\, a,a'\in D.
\end{equation}
When $D=GL(d,\R)$ it is the so-called \emph{the Affine Group of Motions on $\rd$}
\cite{LWWW02}, and corresponds to the action $t\mapsto at+q$.
The  wavelet representation of $\rd\rtimes D$ on $\lrd$ associated with this action is given by
\begin{equation}\label{Wrep} (\nu(q,a)\psi)(t)=|\det a|^{-1/2}\psi(a^{-1}(t-q))=(T_qD_a\psi)(t),\quad t\in\rd.
\end{equation}

\ni
A function $\psi\in\lrd$ is \emph{admissible} if
 the Calder\'on condition is fulfilled:

\begin{equation}\label{hmeas}
\int_D|\hat\psi(\t a \xi)|^2\,da=1, \qquad \text{for a.e.\,
}\xi\in\R^d,
\end{equation}
where  $da$ is a  left  Haar measure on $D$.
The subgroup $D\subset GL(d,\R)$ may be identified
with the  subgroup of $\Spnr$ given by
$$\left\{ \bmatrix a&0\\ 0&\;^t\!a^{-1}\endbmatrix,\quad a\in D\right\},$$
and the metaplectic representation $\mu$ of ${D}$ is
\begin{equation*}
(\mu(a)f)(x)=(\det a)^{-1/2}f(a^{-1}x)=\pm |\det
a|^{-1/2}f(a^{-1}x),\quad f\in L^2(\R^d).
\end{equation*}

The group $\rd\rtimes D$ is isomorphic to the subgroup $H$ of $\rdd\rtimes \Spnr$
given by
\begin{equation}\label{group}
H=\left\{h(q,a)=\left(\bmatrix q\\ 0\endbmatrix,\bmatrix a&0\\ 0&\;^t\!a^{-1}\endbmatrix\right),\,\,q\in\rd, \,\,a\in D \right\}.
\end{equation}
Observe that the group law within $\Spnr$ is $h(q,a)h(q',a')=h(aq'+q,aa')$, in accordance with \eqref{Fpr}.
The extended metaplectic representation restricted to $H$ is
\begin{align}
(\mu_e(h(q,a)f)(t)&=(\rho(q,0)\mu(a)f)(t)=(T_q\mu(a)f)(t)\nonumber\\
&=\pm |\det a|^{-1/2}f(a^{-1}(t-q))=\pm (T_ q D_a f)(t).\label{wavrepest}
\end{align}
Thus, up to a sign, the extended metaplectic representation of $H$ coincides
with the wavelet representation $\nu$, so that they give rise to the same reproducing formula
\begin{equation}\label{wavrep}
f=\int_{D}\intrd\la f, T_q D_a \psi\ra T_q D_a \psi\,\, dh(q,a),
\end{equation}
where $dh$ is a left  Haar measure on $H$. If  $da$ is a left Haar measure of
the group $D$ and $dq$ is the standard Lebesgue measure on $\rd$, a left Haar measure $dh(q,a)$ is given by $dh(q,a)= dq\,|\det a|^{-1}\,da $.

One would expect that also the admissibility conditions related to the two representations coincide.
The next result shows the direct correspondence between them.
\begin{proposition}\label{equiv}
Let $\psi\in L^1(\rd)\cap L^2(\rd)$. Then,
\begin{equation*}
 \int_H W_\psi(h^{-1}\cdot\phas)\,dh=\int_D |\hat{\psi}(^t
 a\xi)|^2\,da.
\end{equation*}
In particular, the wavelet admissibility condition \eqref{hmeas}
and the Wigner one \eqref{admiss} coincide.
\end{proposition}
\begin{proof}
If $\psi\in L^1(\rd)\cap L^2(\rd)$, we have $W_\psi \in \cC(\rdd)$ and  $W_\psi(\cdot,\xi)\in L^1(\rd)$, for every $\xi\in\rd$, and the  Wigner marginal property \eqref{mmm} holds. For
 $h(q,a)^{-1} =h(-a^{-1}q,a^{-1})$ the action of $H$ on the phase space is given by
$$h(q,a)^{-1}\cdot \phas=\bmatrix a^{-1}&0\\0&\t a\endbmatrix\bmatrix x\\ \xi\endbmatrix+
\bmatrix -a^{-1} q\\ 0\endbmatrix= \bmatrix a^{-1}(x-q)\\ \t a
\xi\endbmatrix.
$$
The change of variables
$a^{-1}(x-q)=u$,  $dq=|\det a|\,du$, yields
\begin{align*}
\int_H W_\psi(h^{-1}\cdot\phas)\,dh&=\int_D\intrd W_\psi(a^{-1}(x-q),\t a \xi)\,dq\frac{da}{|\det a|}\\
&=\int_D\intrd  W_\psi(u,\t a \xi)\,du\,da=\int_D |\hat{\psi}(^t
a\xi)|^2\,da,
\end{align*}
that is the claim.
\end{proof}

 Alike  the Gabor case, the assumptions \eqref{suff1} and \eqref{suff2} are not necessary for
a reproducing function. Indeed, it is enough that
$W_\psi(\cdot,\xi)\in L^1(\rd)$, for almost every $\xi\in\rd$.\par

It is worthwhile observing that the wavelet reproducing formula associated to a given subgroup $D$ can be obtained from another subgroup of $G$, namely
\begin{equation}
\tilde{H}=\left\{\tilde{h}(q,a)=\left(\bmatrix 0\\ -q\endbmatrix,\bmatrix ^t\!a^{-1}&0\\ 0&\;a\endbmatrix\right),\,\,q\in\rd, \,\,a\in D \right\}.\label{Giso}
\end{equation}
Indeed, $g=(0,J)\in G$, then $g h(q,a) g^{-1}=\tilde{h}(q,a)$ and $H$  in \eqref{group} and  $\tilde{H}$ are conjugate.
This implies that one is reproducing if and only if the other is, and the corresponding reproducing formulae are equivalent.
Observe that the extended metaplectic representation of $\tilde{H}$  is nothing else but the wavelet representation on the frequency side:
\begin{align}
(\mu_e(\tilde{h}(q,a)f)(t)&=(\rho(0,-q)\mu({}^ta^{-1})f)(t)=(M_{-q}\mu({}^ta^{-1})f)(t)\label{wavrepest2}\\
&=\pm (M_{-q} D_{{}^ta^{-1}} f)(t)=\pm \cF(T_{q}
D_{a}f)(t).\nonumber
\end{align}

\section{The class $\cE$}\label{classeE}
In this section we introduce a class of (lower) triangular subgroups of $\Spnr$ and derive general conditions for reproducing formulae to hold true. From  the
structural point of view, a group $H\in\cE$ is of the form $H= \Sigma \rtimes D$,  where again $D$ is a closed subgroup of $GL(d,\R)$ that acts by
automorphism on $\rd$. Thus, we are given a homomorphism $\theta: D \rightarrow \Aut(\rd)$, that is, a $d$-dimensional real representation of $D$, and we define the semidirect product
\begin{equation}\label{plaw}
h(q,a) h(q',a')=h(\theta(a) q'+q,aa').
\end{equation}
The connection with $Sp(d,\R)$ is as follows. First of all, we shall identify $D$ with
\begin{equation*}\left\{\bmatrix a&0\\ 0&\;^t\!a^{-1}\endbmatrix\,: a\in D\right\}\subset \Spnr.
\end{equation*}
Secondly, we preliminarily observe that
\begin{equation}\label{gen}
 N=\left\{\bmatrix I&0\\ \sigma&\;I
\endbmatrix\,:
\sigma\in \mbox{Sym}(d,\R)\right\}\subset \Spnr
\end{equation}
is an abelian Lie subgroup of $\Spnr$, whereby the matrix product
amounts to $\sigma+\sigma'$, the sum within the vector space of $d$ by $d $ symmetric matrices,
denoted Sym$(d,\R)$.
We then assume that we are given an injective homomorphism of abelian groups
$$j:\rd\rightarrow \mbox{Sym}(d,\R),\quad q\mapsto j(q):=\sigma_q
$$
which establishes a group isomorphism of $\rd$ with the image $\Sigma=j(\rd)\subset$ Sym$(d,\R)$. In other words, we assume that we are given a parametrization of a $d$-dimensional subspace $\Sigma$ of Sym$(d,\R)$. Explicitly:

\begin{equation}\label{matrice2}
\Sigma = \left\{  \bmatrix I&0\\
  \sigma_q&I\endbmatrix\,\,q\in\rd\right\}\subset Sp(d,\R).
\end{equation}

Finally, we consider the products $h=h(q,a)$ defined by
\begin{equation}\label{elem}
h(q,a)= \bmatrix I&0\\  \sigma_q&I\endbmatrix\bmatrix a&0\\
0&\;^t\!a^{-1}\endbmatrix=\bmatrix a&0\\
  \sigma_qa&^t\!a^{-1}\endbmatrix.
\end{equation}
Since
\begin{align*}
h(q,a)h(q',a')&=\bmatrix a& 0\\ \sigma_q a&^t\!a^{-1}\endbmatrix\,\bmatrix a'& 0\\ \sigma_{q'} a'&(^t\!{a'})^{-1}\endbmatrix\\
  &=\bmatrix aa'& 0\\ \sigma_q aa'+^t\!a^{-1}\sigma_{q'}a'&^t\!a^{-1}(^t\!{a'})^{-1}\endbmatrix\\
  &=\bmatrix aa'& 0\\ (\sigma_q +^t\!a^{-1}\sigma_{q'}a^{-1})aa'&^t\!(aa')^{-1}\endbmatrix\\
\end{align*}
and
$$
h(\theta(a) q'+q,aa')= \bmatrix aa'&0\\
  \sigma_{\theta(a) q'+q}(aa')&^t\!(aa')^{-1}\endbmatrix,
$$
 the semidirect product law \eqref{plaw} holds true if and only if the parametrization $j$ and the representation $\theta$ satisfy
\begin{equation}\label{equmemee}
\sigma_{\theta(a) q}=^t\!a^{-1}\,\sigma_q \,a^{-1}, \quad a\in
D,\,\,\, q\in\rd.
\end{equation}

Formally, a group $\cE$ can thus be described by a triple $(D,j,\theta)$, where $D$ is a closed subgroup of $GL(d,\R)$, $\theta: D \rightarrow \Aut(\rd)$ is a representation, and $j:\rd\rightarrow \mbox{Sym}(d,\R)$, is an injective homomorphism of abelian groups; the data must satisfy the compatibility equation \eqref{equmemee}. We avoid this excess of notation and write directly $H=\Sigma\rtimes D$. If $H\in\cE$ is chosen, and hence $D$, we assume that a  left Haar measure  $da$ on $D$ has been fixed and consequently the left Haar measure on $H$ that will be fixed is
$$dh=dq\frac{da}{
|\det \theta(a)|},$$
 where $dq$ is the Lebesgue measure on $\rd$.
Finally, the metaplectic representation restricted to $H\in\cE$ is given by
 \begin{align}
\mu(h(q,a))f(x)&=\mu\left(\bmatrix I&0\\  \sigma_q&I\endbmatrix\right)\mu\left(\bmatrix a&0\\
0&\;^t\!a^{-1}\endbmatrix\right) f(x) \nonumber\\
  &=\pm e^{\pi i\langle
\sigma_q x,x\rangle}\mu\left(\bmatrix a&0\\
0&\;^t\!a^{-1}\endbmatrix\right) f(x)\nonumber\\
  &=\pm e^{\pi i\langle
\sigma_q x,x\rangle}(\det a)^{-1/2}f(a^{-1}x).\label{metap}
\end{align}

\subsection{Examples}
\begin{enumerate}
    \item Let $d=2$.   The Translation-Dilation-Sheering group (TDS) is the $4$-dimensional triangular
reproducing group introduced and studied in \cite{AEFK1}. It is defined by
\begin{equation}\label{Baby}
TDS=\Bigl\{ \bmatrix tS_{\ell}&0
\\t B_y S_{\ell}&t^{-1} {}\t S_{\ell}^{-1}\endbmatrix
:t>0,\,\ell\in\R,\;y\in\R^2\Bigr\}
\end{equation}
 where if $y=(y_1,y_2)\in\R^2$ and $\ell\in\R,$
\begin{equation*}
B_y=\bmatrix 0&y_1\\ y_1&y_2\endbmatrix,
\quad S_\ell=\bmatrix 1&\ell\\0&1\endbmatrix.
\end{equation*}
It is isomorphic to the semidirect product $\Sigma\rtimes D$, with $\Sigma=\{\sigma_y, y\in\R^2\}$
$D=\{ t S_{\ell},\,\,t>0,\,\ell\in\R\}$.
A simple computation shows that
$$ \theta(tS_{\ell}) y={}\t(tS_{\ell})^{-2}y.$$
 Thus  $TDS\in \cE$. The group $TDS$ is important because it is the group underlying shearlet theory (see e.g. \cite{Dalke, guo-labate}). This terminology stems from the geometric action of $S_\ell$ on $\R^2$, known as shearing transformation.
\item In dimension $d=2$, consider the group $SIM(2)$  given by
 \begin{equation}\label{sim2g}
 \left\{h(t,y,\varphi)=\bmatrix tR_{\varphi}&0
\\t\Sigma_y R_{\varphi}&t^{-1}R_{\varphi}\endbmatrix,\quad t>0,y\in\R^2,\varphi\in[0,2\pi)\right\},
 \end{equation}
where $R_\varphi= \bmatrix
\cos\varphi&\sin\varphi\\-\sin\varphi&\cos\varphi\endbmatrix$ and  $\Sigma_y=\bmatrix y_1&y_2\\y_2&-y_1\endbmatrix$.
 This subgroup of $Sp(2,\R)$ is also reproducing \cite{ALI2000,AEFK1} and its semidirect structure is $\Sigma\rtimes D$, with $\Sigma=\{\sigma_y, y\in\R^2\}$ and
$D=\{ t R_{\varphi},\,\,t>0,\,\varphi\in[0,2\pi)\in\R\}$. In this case, we have
$$\theta(tR_{\varphi}) y={}\t(tR_{\varphi})^{-2}y,\quad t>0, \,\varphi\in[0,2\pi),\,y\in\R^{2}.
$$
The $SIM(2)$ group is is named so because it is isomorphic to the group of similitude transformations of the plane. It is one in the family of groups
$H_{\alpha,\beta}$ introduced and studied in \cite[Section
6]{AEFK1}. They display an analogous semidirect product structure, and they all belong to $\cE$.
\end{enumerate}
\subsection{The mapping $\Phi$.}\label{42}
Much of  the analysis of the metaplectic representation on a group in the class $\cE$ originates  from the properties of a fundamental quadratic mapping of $\rd$, whose basic properties are described in the next proposition.
\begin{prop}\label{propke}
There exists a quadratic mapping   $\Phi:\rd\rightarrow\rd$, that satisfies
\begin{align}
 \la\sigma_q x,x\ra&=-2\la q, \Phi(x)  \ra, \label{diff2}\\
 \Phi(a^{-1}x)&= {}^t\theta(a)\Phi(x),\label{fuori}
\end{align}
for every $ x\in\rd$ and every $ a\in D.$
The mapping  $\theta$ is defined in   \eqref{equmemee}.
\end{prop}
\begin{proof} We shall compute explicitly
the mapping $\Phi$ and prove \eqref{diff2} and \eqref{fuori}.
Select a basis $\{e_i\}$ of $\rd$ and put $\sigma^i=j(e_i)$. Thus, if $q=\sum_{i=1}^d q_i e_i$, then $\sigma_q=\sum_{i=1}^d q_i \sigma^i$ and
\begin{equation*}
\langle \sigma_q x,x\rangle =\langle (\sum_{i=1}^d q_i \sigma^i)x,x\rangle=\sum_{i=1}^d q_i\langle \sigma^ix,x\rangle=-2\la q,\Phi(x)\ra,
\end{equation*}
where
\begin{equation}\label{Phi}
   \Phi(x)=(\Phi_1(x),\dots,\Phi_d(x)),\quad\mbox{with}\quad \Phi_j(x)=-\frac12\langle \sigma^j
   x,x\rangle.
\end{equation}
This establishes  \eqref{diff2}. Finally, using  \eqref{equmemee} and \eqref{diff2}, we obtain
\begin{align*}
-2\la q,\Phi(a^{-1}x)\ra &=\langle \sigma_q a^{-1}x,a^{-1}x\rangle= \langle {}^t a \sigma_{\theta(a)q} x,a^{-1}x\rangle\\
&=\langle \sigma_{\theta(a)q} x,a a^{-1}x\rangle
=\langle \sigma_{\theta(a)q} x,x\rangle\\
&=-2\la \theta(a) q,\Phi(x)\ra=-2\la  q,{}^t\theta(a) \Phi(x)\ra,
\end{align*}
 for every $q\in\rd$, hence  equality \eqref{fuori}.
\end{proof}

 Now, we would expect that the reproducing formula coming from the metaplectic representation of $H$ in \eqref{metap} is somehow
  equivalent to the wavelet case \eqref{wavrepest}. This is what we
  are going to exhibit.

 Our approach gives a general criterion for reproducibility which contains those  in \cite{AEFK1,AEFK2}.

 Let $\Phi\,:\, \R^d\rightarrow \R^d$ be a
 quadratic mapping
and assume that $J_{\Phi}$, the
Jacobian of $\Phi$, does not vanish
identically. Set
$$S=\{x\in \R^d\,:\,
J_{\Phi}(x)=0\}.
$$ Notice that, if $y_0\in
\R^d\setminus\Phi(S)$ and $x_0\in \Phi^{-1}(y_0)$, then $x_0\notin
S$. Hence, by the local invertibility Theorem there exists an open
neighborhood $A$  of $x_0$ and an open neighborhood $B$ of $y_0$
such that $\Phi_{|_A}\,:A\rightarrow B$ is a diffeomorphism. The
local invertibility theorem in $\mathcal{C}^d$ also tells us that
in a neighborhood of $x_0$ in $\mathcal{C}^d$ there are no other
solutions of $\Phi(x)=y_0$. Hence, being the solutions of
$\Phi(x)=y_0$ isolated, by Bezout Theorem they are at most $2^d$
(see (\cite[Section 2.3, page 10]{fulton2}, or also \cite{GaLi80}).

To sum up, for any open set
$X\subset\mathbb{R}^n$ contained in
$\Phi(\R^d)\setminus \Phi(S)$ such that
the cardinality of $\Phi^{-1}(y)$ is
locally constant for $y\in X$, the
mapping $\Phi$ induces a surjective
finite-sheeted covering
$$\Phi\,:\,\Phi^{-1}(X)\longrightarrow
X.
$$
(observe that $\Phi^{-1}(X)$ is also open). In other terms, every
$y_0\in X$ has an open neighborhood $B$ such that
$\Phi^{-1}(B)=\cup_{j=1}^{k} A_j $, $k\leq 2^d$, $A_j\subset
\R^d\setminus S$, $A_j$ open and $A_j\cap A_i=\emptyset$ if
$i\not=j$, and $\Phi\,:\, A_j\rightarrow B$ is a diffeomorphism.
%
\vskip0.5truecm
\textbf{Remark} It could be useful to observe that the assumption
that the cardinality of $\Phi^{-1}(y)$ for $y\in X$ is locally
constant (constant if $X$ is connected) is in fact equivalent to
requiring the more easy to check hypothesis that the map $\Phi:
\Phi^{-1}(X)\to X$ is proper, i.e. for every compact $K\subset X$,
$\Phi^{-1}(K)$ is compact.\par Indeed, assume that the cardinality
of $\Phi^{-1}(y)$ for $y\in X$ is locally constant. Then, as we
saw, $\Phi: \Phi^{-1}(X)\to X$ is a finite-sheeted covering, and
therefore it is clear that $\Phi^{-1}(K)$ is a compact subset,
when $K$ is contained in one of the open subset $B$ above. In the
case of a general compact $K\subset X$, let $x_n\in\Phi^{-1}(K)$
be a sequence. Possibly after replacing $x_n$ with a subsequence,
$\Phi(x_n)$ will converge to an element $\overline{y}\in K$, so
that it must belong to a compact subset $K'$ contained in an a
small neighbourhood $B$ of $\overline{y}$. Hence $x_n\in
\Phi^{-1}(K')$, which is compact by what we have just observed, so
that still a sequence of $x_n$ should converge to an element
$\overline{x}$. By continuity, $\Phi(\overline{x})=\overline{y}$,
so that $\overline{x}\in\Phi^{-1}(K)$.\par Viceversa, suppose that
the map $\Phi: \Phi^{-1}(X)\to X$ is proper. Take $\overline{y}\in
X$ and let $\Phi^{-1}(\overline{y})=
\{\overline{x}_1,...,\overline{x}_k\}$. We already know that, for
convenient pairwise disjoint open neighbourhoods $A_j$,
$j=1,...,k$ of $x_j$ and for an open neighbourhood $B$ of
$\overline{y}$, $\Phi:A_j\to B$ is a diffeomorphism, so that every
$y$ sufficiently close to $\overline{y}$ has at least $k$
pre-images, each contained in one of the $A_j$'s. Now suppose, by
contradiction, that there exists a sequence $y_n\to \overline{y}$,
with each $y_n$ having a further pre-image $x_n$; hence
$x_n\not\in\cup_{j=1}^k A_j$. By the hypothesis of properness, one
subsequence of $x_n$ must converge to an element
$\overline{x}\not\in\cup_{j=1}^k A_j$. By continuity we have
$\Phi(\overline{x})=\overline{y}$, so that $\overline{x}\in
\Phi^{-1}(\overline{y})\subset \cup_{j=1}^k A_j$, which is a
contradiction.

With the notation above, we
have the following result.

\begin{prop}\label{Fabio}
Suppose $X\subset\mathbb{R}^d$ is an
open, simply connected set, contained
in $\Phi(\R^d)\setminus \Phi(S)$, such
that the cardinality of $\Phi^{-1}(y)$
is constant for $y\in X$. Then there
exists an integer $k\in
\{1,\dots,2^d\}$ and there exist open,
connected, and pair-wise disjoint sets
$Y_1,\dots,Y_k\subset \R^d\setminus S$
such that $\Phi^{-1}(X)=\cup_{j=1}^k
Y_j$, and
$$\Phi_{|_{Y_j}}\,:\,Y_j\rightarrow X\,\, \mbox{is\, a\, diffeomorphism}.
$$
\end{prop}
\begin{proof} Since $\Phi\,:\, \Phi^{-1}(X)\rightarrow X$ is a covering and $X$ is open (hence locally path-connected) and simply connected, it follows (see, e.g., \cite[Corollary 13.8]{Fulton}) that the covering is trivial: i.e., there exists an homeomomorphism  $\Psi: \,X\times\Phi^{-1}(y_0)\rightarrow \Phi^{-1}(X)$, $y_0$ being any fixed
point in $X$. Since the cardinality of
$\Phi^{-1}(y_0)$ is at most $2^d$, we
have
$$\Phi^{-1}(y_0)=\{x_1,\dots,x_k \}\subset\R^d\setminus S,\quad k\leq 2^d,
$$
and the desired result follows by taking $Y_j=\Psi(X\times \{ x_j\})$.
\end{proof}

Observe that, since $\Phi$ is even, if $\Phi$ is a diffeomorphism from  $Y_j$ onto $X$, then it is a diffeomorphism
from $-Y_j$ onto $X$. Hence the number $k$ of $Y_j$ is even.

\textbf{Examples.}
\begin{enumerate}
    \item In dimension $d=2$,   consider the mapping $\Phi$, related to the  $TDS$ group in \eqref{Baby} and defined by (see \cite[(5.15)]{AEFK1})
    $$\Phi(x_1, x_2)=\left(-x_1x_2,-\frac{x_2^2}2 \right).$$
Here $S=\{(x_1,0), x_1\in\R\}$,
$\Phi(S)=(0,0)$ and
$\Phi(\R^2)=\R\times\R_-\cup\{(0,0)\}$.
If we define $X=\R\times\R_-$, then $X$
is open and simply connected and
$$\Phi^{-1}(X)=Y_1\cup Y_2$$
where
$$Y_1=\{(x_1,x_2)\in\R^2\,:\, x_2>0 \},\quad Y_2=-Y_1=\{(x_1,x_2)\in\R^2\,:\,  x_2<0 \}.
$$
\item In dimension $d=2$, consider the mapping $\Phi$, related to the  $SIM(2)$ group in \eqref{sim2g} and defined by (see \cite[(5.11)]{AEFK1})
$$\Phi(x_1, x_2)=\Big(\frac{x_2^2-x_1^2}2, -x_1x_2\Big).$$
Here $S=\{(0,0)\}$, $\Phi(S)=(0,0)$ and $\Phi(\R^2)=\R^2$.  If we define $X=\R^d\setminus\{(x_1,0),x_1\leq 0\}$, then $X$ is open and simply connected and
$$\Phi^{-1}(X)=Y_1\cup Y_2$$
where
$$Y_1=\{(x_1,x_2)\in\R^2\,:\, x_2>0 \},\quad Y_2=-Y_1=\{(x_1,x_2)\in\R^2\,:\, x_2<0 \}.
$$
\item In dimension $d\geq 2$,   consider the mappings $\Psi_p$, related to the  group $F=\H_e^{d-1}\rtimes U(d-1)\subset\Spnr$, with $\H_e^{d-1}$  the Heisenberg group extended by the usual $1$-dimensional  homogeneous dilations, studied  in \cite{AEFK2} and defined by (see \cite[(20)]{AEFK2})
    $$\Psi_p(x', x_d)=\left(x_dx'-\frac{1}{2}x_d^2\,p, \frac{1}{2}x_d^2\right),\quad (x',x_d)\in\R^{d-1}\times\R,$$
    where $p\in\R^{d-1}$ is fixed.
Here $S=\{(x',0), x'\in\R^{d-1}\}$,
$\Phi(S)=(0,0)$ and
$\Phi(\R^d)=\R^{d-1}\times\R_+\cup\{(0,0)\}$.
If we define $X=\R^{d-1}\times\R_+$,
then $X$ is open and simply connected
and
$$\Phi^{-1}(X)=Y_1\cup Y_2$$
where
$$Y_1=\{(x',x_d)\in\R^{d-1}\times\R\,:\, x_d>0 \},\quad Y_2=-Y_1=\{(x',x_d)\in\R^{d-1}\times\R\,:\,  x_d<0 \}.
$$

\item In dimension $d=2$, consider the mapping $\Phi$, related to
the TDH group, defined in the subsequent \eqref{fil}, given by
(see \eqref{PhiF})
$$
\Phi(x_1,x_2)=(-\frac12(x_1^2+x_2^2),-x_1x_2).
$$
 We have
 $$S=\{(x_1,x_2)\in\R^2\,:\, x_1=\pm x_2\},\quad \Phi(S)=\{(x_1,x_2)\in\R^2\,:\, x_1=\pm x_2, x_1\leq 0\},$$
If we set
\begin{equation}\label{xset}
X=\{(u,v)\in\R_-\times\R\,:\,u^2-v^2>0\}, \end{equation}
then
$$\Phi(\R^2) =X\cup\{(0,0)\}.$$
In this case $k=2^2=4$ and
$$\Phi^{-1}(X)=Y_1\cup Y_2\cup Y_3\cup Y_4,$$
where
\begin{equation}\label{ysets1}Y_1=\{(x_1,x_2)\in\R_-\times\R\,:\,
x_1^2-x_2^2>0 \},\quad Y_2=-Y_1=\{(x_1,x_2)\in\R_+\times\R\,:\,
x_1^2-x_2^2>0 \},
\end{equation}
\begin{equation}\label{ysets2}Y_3=\{(x_1,x_2)\in\R\times\R_-\,:\, x_1^2-x_2^2<0 \},\quad
Y_4=-Y_3=\{(x_1,x_2)\in\R\times\R_+\,:\, x_1^2-x_2^2<0 \}
\end{equation}
(see Figure 1).
\end{enumerate}

  \vspace{1.2cm}
  \small
 \begin{center}
           \includegraphics{Filgroup.1}
            \\
           $ $
           \begin{center}{Figure 1:
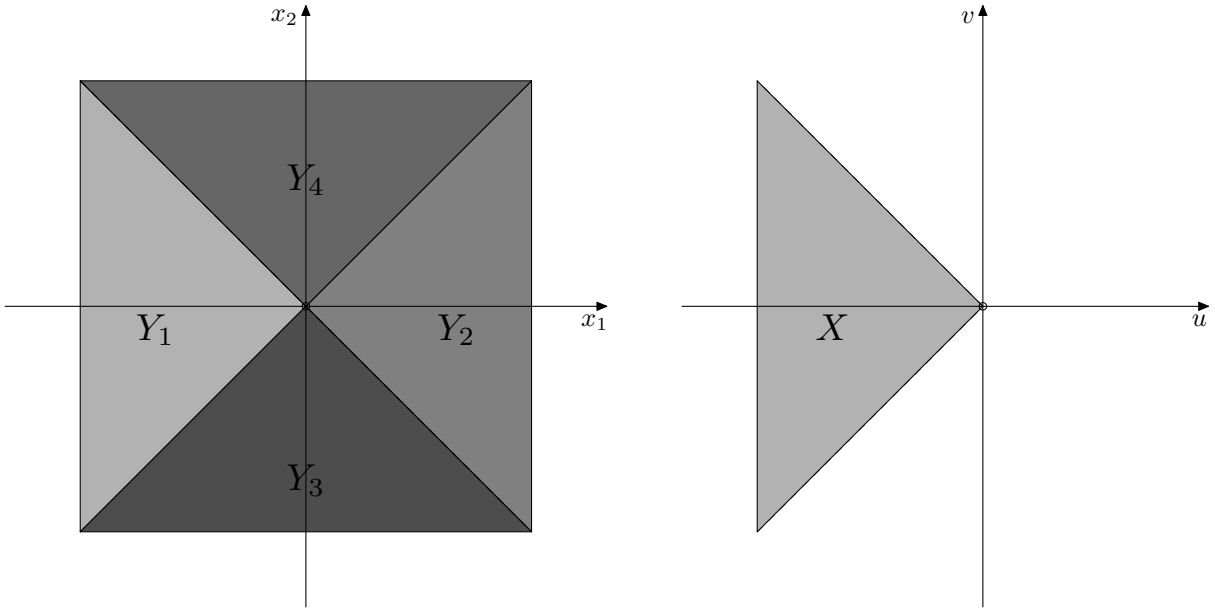
           \small The sets $X$ and $Y_j$, $j=1,\dots 4$, for the TDH group.}
           \end{center}
\end{center}
\normalsize
 \vskip0.5truecm

\begin{lemma}\label{locplanch} Let  $\Phi$,  $Y_j, X$ as in Proposition \ref{Fabio}. If $h\in\cC^\infty_0(Y_j)$, then \begin{equation}\label{planchA}
\int_{\R^d}\left|\int_{Y_j} h(x) e^{2\pi i \langle
q,\Phi(x)\rangle}\,dx\right|^2\,dq=\int_{Y_j}|h(x)|^2\frac{dx}{|J_{\Phi}(x)|},\end{equation}
where $J_{\Phi}(x)$ is the Jacobian of $\Phi$ at $x$.
\end{lemma}
\begin{proof} Recall that $Y_j\subset \R^d\setminus S$, so that $J_{\Phi}(x)\not=0$ on $Y_j$ and $\Phi_j\,:= \Phi_{|_{Y_j}}$ is a diffeomorphism from $Y_j$ onto $X$. This let us  make the change of variables $\Phi(x)=u$
and use Plancherel's formula:
\begin{align*}
\int_{\R^d}\left|\int_{Y_j} h(x) e^{2\pi i \langle
q,\Phi(x)\rangle}\,dx\right|^2\,dq &=\int_{\R^d}\left|\int_{X}
h(\Phi_j^{-1}(u))e^{2\pi i \langle
q,u\rangle}\,|J_{\Phi_j^{-1}}(u)|\,du\right|^2\,dq\\
&=\int_{\R^d}\left|\int_{\R^d}\chi_{X}(u) h(\Phi_j^{-1}(u))e^{2\pi i
\langle
q,u\rangle}\,|J_{\Phi_j^{-1}}(u)|\,du\right|^2\,dq\\
&=\int_{\rd}|\chi_X(u)h(\Phi_j^{-1}(u))J_{\Phi_j^{-1}}(u)|^2\,du,\\
&=\int_{X}|h(\Phi_j^{-1}(u))J_{\Phi_j^{-1}}(u)|^2\,du,\\
&=\int_{Y_j}|h(x)|^2 \frac{dx}{|J_{\Phi}(x)|},
\end{align*}
where in the last raw we have performed the change of variables
$\Phi_j^{-1}(u)=x$. Observe that, since the supp$\, h$ is a
compact set contained in the open set $Y_j$, there exist two
constants $0<c<C$, such that $c<|J_{\Phi}(x)| <C$ on supp$h$ and the last integral is well defined.
\end{proof}


Now we have all the pieces in place to provide the reproducing
condition on the wavelet $\psi$  which guarantees the
reproducibility of the group $H$.

\begin{theorem}\label{RepThm} Let $H= \Sigma\rtimes D\cong\rd\rtimes D$ be as at the beginning of this section
 and let  $X,Y_j$ be as in Proposition \ref{Fabio}. Then, the identity
\begin{equation}\label{repc}\|f\|^2_2 =\int_H \, |\langle f,\mu(h(q,a))\psi\rangle|^2\,dh(q,a)
\end{equation}
holds for every $f\in
\cC^{\infty}_0(Y_j)$ if and only if
$\psi$ satisfies the condition
\begin{equation}\label{co1-0}
\int_D |\psi(a^{-1} x)|^2 \,|\det
\theta(a)a|^{-1}\,da=|J_{\Phi(x)}|,\qquad
a.e.\,\,x\in Y_j.
\end{equation}
Moreover, \eqref{repc} holds for every
$f\in \cC^{\infty}_0(Y_j\cup(-Y_j))$ if
and only if $\psi$ satisfies the
following two conditions:
\begin{equation}\label{co1}
\int_D |\psi(a^{-1} x)|^2 \,|\det
\theta(a)a|^{-1}\,da=\int_D |\psi(-a^{-1} x)|^2 \,|\det
\theta(a)a|^{-1}\,da=|J_{\Phi(x)}|,
\end{equation}
for $a.e.\,\,x\in Y_j$, and
\begin{equation}\label{co2}
\int_D \psi(-a^{-1}x)\overline{\psi(a^{-1} x)} \,|\det
\theta(a)a|^{-1}\,da=0,\quad a.e.\,\,x\in Y_j.
\end{equation}
\end{theorem}
\begin{proof}
The left Haar measure on $H$ is given by $$dh(q,a)=dq |\det \theta(a)|^{-1}da$$
 and  the metaplectic representation on $H$
in \eqref{metap}, let us  write, for every $f\in \cC^{\infty}_0(Y_j\cup(-Y_j))$,
\begin{align}\label{xxxx}
\int_{H}\, |\langle f,\mu(h(q,a))\psi\rangle|^2\,dh(q,a)
&=\int_D\int_{\rd}\left|\int_{\R^d} f(x)  e^{-\pi i \langle
\sigma_q x,x\rangle}(\det a)^{-1/2}\overline{\psi(a^{-1}
x)}\,dx\right|^2dq \\
&\quad\quad\quad\quad\quad\cdot|\det \theta(a)|^{-1}\,da\nonumber\\
&=\int_D\int_{\rd}\!\left|\int_{\R^d}\! f(x)  e^{2\pi i \langle
q,\Phi(x)\rangle}\overline{\psi(a^{-1} x)}\,dx\right|^2\!dq |\det
 \theta(a)a|^{-1}\!da\nonumber\\
 &=\int_D\int_{\rd}\!\left|\int_{Y_j} [f(x)\overline{\psi(a^{-1} x)}+f(-x)\overline{\psi(-a^{-1} x)}] \nonumber\right.\\
 &\quad\quad\quad\quad\quad\cdot\left. e^{2\pi i \langle
q,\Phi(x)\rangle}\,dx\right|^2\!dq |\det
 \theta(a)a|^{-1}\!da\nonumber
\end{align}
where the last equality is due to the even property of $\Phi$. \par
Now, we set $h(x):=f(x)\overline{\psi(a^{-1} x)}+f(-x)\overline{\psi(-a^{-1} x)}$ and apply Lemma \ref{locplanch}, so that
\begin{align}
\int_{H}& |\langle f,\mu(h(q,a))\psi\rangle|^2\,dh(q,a)
=\int_D\int_{Y_j}[| f(x) \overline{\psi(a^{-1}x)}|^2+| f(-x) \overline{\psi(-a^{-1}x)}|^2\nonumber\\
&\,\quad \quad\quad\quad\quad\quad + 2{\mathcal Re}\,(f(x)
\overline{\psi(a^{-1}x)} \overline{f(-x)} \psi(-a^{-1}x)]
\frac{dx}{|J_{\Phi(x)}|}|\det  \theta(a) a |^{-1}\!da\nonumber.
\end{align}

Suppose at first that $f$ satisfies the additional property:
$f(x)=0$ on $-Y_j$, then
\begin{equation*}
\int_{H} |\langle f,\mu(h(q,a))\psi\rangle|^2\,dh(q,a)
=\int_{Y_j}|f(x)|^2\left(\int_D|\psi(a^{-1}x)|^2\,|\det
 \theta(a) a |^{-1}\,da\right)\frac{dx}{|J_{\Phi(x)}|}
\end{equation*}
so that $\int_{H} |\langle
f,\mu(h(q,a))\phi\rangle|^2\,dh(q,a)=\|f\|^2_2$
if and only if \eqref{co1-0} holds.

If, instead,  $f(x)=0$ on $Y_j$, the equality  $\int_{H} |\langle
f,\mu(h(q,a))\phi\rangle|^2\,dh(q,a)=\|f\|^2_2$ holds true if and only if the second and the last expression in \eqref{co1} are equal.

Finally, taking  $f\in\cC_0^\infty(Y_j\cup(-Y_j))$,  such that $f(x)\overline{f(-x)}$ is real-valued,
 purely imaginary-valued, respectively, we have
\begin{align*}
\int_{H} &|\langle f,\mu(h(q,a))\phi\rangle|^2\,dh(q,a)=\|f\|_2^2
\end{align*}
if and only if both  conditions \eqref{co1} and \eqref{co2} are fulfilled.
\end{proof}

\begin{cor} Theorem \ref{RepThm} still holds if the
assumptions $f\in \cC^\infty_0(Y_j)$
or $f\in\cC^\infty_0(Y_j\cup (-Y_j))$
are replaced by
 $f\in L^2(Y_j)$, or $f\in L^2(Y_j\cup (-Y_j))$, respectively.
 \end{cor}
\begin{proof} It follows by the density
of $\cC^\infty_0(Y_j)$ and $\cC^\infty_0(Y_j\cup (-Y_j))$
 in $L^2(Y_j)$ and $L^2(Y_j\cup (-Y_j))$, respectively.
\end{proof}
\vskip0.5truecm

\subsection{The case $H= \rd\rtimes N$ }
We shall exhibit that for subgroups of the kind $H= \rd\rtimes
N$, with  $N$ defined in \eqref{gen},
formula \eqref{rep} ever fails. The product law is given by
$h(q,n)h(q',n')=h(q+nq',nn')$.

First of all, in the semidirect product above, we are dealing with the case  $\rd\cong \{0\}\times\rd\subset\rdd$. The choice  $\rd\cong \rd \times \{0\}$  reduces $\rd\rtimes N$ to $\rd$, since it forces $N$ to be $I$.
In our case, for $p,p'\in\rd$, $n=\bmatrix I& 0\\ c &I\endbmatrix$, with ${}^t c=c\in M(d,\R)$, the action on $\rd$ is
\begin{equation*}
p+np'=\bmatrix0\\p\endbmatrix+\bmatrix I& 0\\ c &I\endbmatrix\,\bmatrix0\\p'\endbmatrix
  =\bmatrix 0\\ p+p'\endbmatrix,
\end{equation*}
and the product law becomes
$$h(p,c)h(p',c')=h(p+p',c+c').
$$
The extended metaplectic representation on $H$ is given by
$$\mu_e(h(p,n))f(t)=\rho(0,p)\mu(n)f(t)=\pm M_pe^{i\pi\la ct,t\ra}f(t).
$$
The right-hand side of  \eqref{rep} has the form
\begin{align*}
\int_H|\la f,\mu_e(h)\f\ra|^2\,dh&=\int_{N}\int_{\rd}\left|\cF\left(e^{i\pi\la c\cdot,\cdot\ra}\bar{\f}f\right)(p)\right|^2dp\,d\mu(c)\\
&=\int_{N}\int_{\rd}\left|e^{i\pi\la ct,t\ra}f(t)\overline{\f(t)}\right|^2\,dt\,d\mu(c)\\
&=\int_{N}d\mu(c)\int_{\rd}|f(t)|^2 |\f(t)|^2\,dt,
\end{align*}
where we used Plancherel's formula, so that the last integral either vanishes or diverges.

\section{New $2$-dimensional reproducing subgroups $H= \Sigma\rtimes D$}

We shall construct two new examples of reproducing subgroups $H= \Sigma\rtimes D$, in dimension $d=2$. To prove their reproducibility, we shall apply the theory developed in the previous section.

\subsection{The TDH group}
Consider the $4$-dimensional group:
\vskip0.1truecm
\begin{equation}\label{fil}
TDH=\Bigl\{h((x,y), (s,t)):= \bmatrix
e^{-s}H(t)&0\\e^sT(x,y)H(t)&e^{s}H(-t)\endbmatrix
:s,t,x,y\in\R\Bigr\}\subset \Sptwor,
\end{equation}
\vskip0.4truecm \noindent with the hyperbolic matrix $H(t)$  given by
\begin{equation}\label{hyp}
H(t)=\bmatrix \cosh t&\sinh t\\\sinh t&\cosh t\endbmatrix,\quad t\in\R
\end{equation}
whereas the symmetric matrix $T(x,y)$ displays the entries
\begin{equation}\label{matT}
T(x,y)=\bmatrix x&y\\y&x\endbmatrix,\quad x,y\in\R.
\end{equation}
The semidirect structure  $H= \R^2\rtimes D$ is clear:
$$\bmatrix
e^{-s}H(t)&0\\e^sT(x,y)H(t)&e^{s}H(-t)\endbmatrix=\bmatrix I
&0\\T(x,y)&I\endbmatrix\bmatrix
e^{-s}H(t)&0\\0&e^{s}H(-t)\endbmatrix.
$$
\begin{proposition}\label{propF} The subgroups $TDH$ of $\Sptwor$  satisfy
the following properties:
\begin{itemize}
\item[(a)] The  product law in $TDH$, is explicitly given by:
$$
h(s,t,z)h(s',t',z')
=h(s+s',t+t',z+e^{2s}H(-2t)z'),\quad z,z'\in\R^2,
t,t',s,s'\in\R.
$$
\item[(b)] The left Haar measure on $TDH$ is $d
h(s,t,z)=e^{-4s}\,ds\,dt\,dz$. \item[(c)] The mapping $\Phi$ in
\eqref{Phi} is explicitly given by
\begin{equation}\label{PhiF}
\Phi(x)=(-\frac12(x^2+y^2),-xy)
\end{equation}
and has  Jacobian $J_{\Phi}(x,y)=-(x^2-y^2)$. Observe that
$\Phi(\R^2)=X\cup\{(0,0)\}$, where the open set $X$ is defined in
\eqref{xset}.
 \item[(d)] The restriction of the metaplectic
representation to $TDH$ is given by:
\begin{equation}\label{mualfabeta}
\mu(h((s,t),(x,y)))f(u)=\pm e^{ s} e^{\pi i \langle
T(x,y) u,u\rangle} \,f(e^{s}H(- t)u),\quad f\in L^2(\R^2)
\end{equation}
\item[(e)] The group homomorphism $\theta$ in \eqref{plaw} is
\begin{equation}\label{thetaF}\theta(e^{-s} H(t))=e^{2s}H(-2t)=(e^{-s}H(t))^{-2}={}^t(e^{-s}H(t))^{-2},
\end{equation}
since the matrix $a=a(s,t)=e^{-s} H(t)$ is symmetric. Hence
$\theta$ is the same homomorphism encountered in the $TDS(2)$ and
$SIM(2)$ group cases.
\end{itemize}
\end{proposition}
Since the proof consists of easy calculations, we leave it to the interested reader.

The reproducibility of the  group $TDH$ is then a mere application
of Theorem \ref{RepThm}, with $\Phi^{-1}(X)=Y_1\cup Y_2\cup
Y_3\cup Y_4,$ and $X$ defined in \eqref{xset}, $Y_1,Y_2=-Y_1$
defined in \eqref{ysets1} and $Y_3, Y_4=-Y_3$ defined in
\eqref{ysets2}.
\begin{theorem} The subgroup $TDH$ is reproducing for $L^2(Y_1\cup-(Y_1))$. Moreover,   $\psi\in L^2(Y_1\cup-(Y_1))$ is a reproducing function for $TDH$
if and only if
\begin{equation}\label{condF1}
\int_{Y_1} |\psi(u,v)|^2 \,\frac{du\,dv}{(u^2-v^2)^2}\,=\int_{Y_1} |\psi(-u,-v)|^2 \,\frac{du\,dv}{(u^2-v^2)^2}\,=1,
\end{equation}
and
\begin{equation}\label{condF2}
\int_{Y_1} \psi(-u,-v)\overline{\psi(u,v)} \,\frac{du\,dv}{(u^2-v^2)^2}\,= 0.
\end{equation}
Similarly, $TDH$ is  reproducing for $L^2(Y_3\cup-(Y_3))$,  and
$\psi\in L^2(Y_3\cup-(Y_3))$ is a reproducing function if fulfills
\eqref{condF1} and \eqref{condF1} with $Y_3$ in place of $Y_1$.
\end{theorem}
\begin{proof}
We use Theorem \ref{RepThm} and  translate conditions \eqref{co1} and   \eqref{co2} into this context. The automomorphism $\theta$ is computed in  \eqref{thetaF}, whence
$$|\det\theta(a)a|=|\det((e^{-s}H(t))^{-2}e^{-s}H(t))|=|\det(e^{s}H(-t))|^{-1}|=e^{-2 s}.$$
We consider the case $\Phi=\Phi_{|_{Y_1}}$. The first condition in \eqref{co1} reads in this framework as
$$\int_{\R^2} |\psi(e^{s}H(-t){}^t(x,y))|^2 e^{-2s}\,ds\,dt=x^2-y^2\,\,\quad \mbox{a.e.}\,(x,y)\in Y_1.
$$
 Performing the change of variables $e^{s}H(-t){}^t(x,y)={}^t(u,v)$, we get $e^{2s}(x^2-y^2)=u^2-v^2$.  Here
 $ds\,dt=\displaystyle\frac{1}{u^2-v^2}\,du\,dv$. Hence, the previous integral coincides con the first one in
 the left-hand side of \eqref{condF1}. The other cases are analogous.
\end{proof}

\subsection{The TDW group}

The TDW group arises by tensor-product of $1$-dimensional wavelets (see the end of this section) and  is defined as follows.
$$
H=\Bigl\{h((x,y), (s,t))= \bmatrix
e^{s}&0&0&0\\0&e^t&0&0\\e^{s}x&0&e^{-s}&0\\0&e^ty&0&e^{-t}\endbmatrix :s,t,x,y\in\R\Bigr\}\subset \Sptwor.
$$
The TDW group  enjoys the following properties:\\
$(i)$ If we set $a(s,t)=\bmatrix e^s&0\\0&e^t \endbmatrix$, the
product law in $ H$ is explicitly given by:
$$
h(s,t,z)h(s',t',z')
=h(s+s',t+t',z+a(s,t)^{-2}z'),\quad z,z'\in\R^2,
t,t',s,s'\in\R.
$$ Hence the authomorphism $\theta$ is $ \theta(a(s,t))=a(s,t)^{-2}$ (notice that $(a(s,t)$ is symmetric).\\
$(ii)$ The mapping $\Phi$ in \eqref{Phi}
is  given by
$$
\Phi(x)=-\frac12 (x^2,y^2)
$$
and has  Jacobian $J_{\Phi}(x,y)=xy$, so that $J_{\Phi}(x,y)= 0$
on the set $S=\{(x,y)\,:\,x=0\vee y=0\}$ and
$\Phi(S)=\{(x,0),\,x\leq 0\}\cup\{(0,y),\,y\leq0\}$. Moreover,
$\Phi(\R^2)=\{(x,y)\,:\,x\leq0,y\leq0\}$. In this case, we have
$$ X=\R_-\times\R_-,\quad\Phi^{-1}(X)=Y_1\cup(-Y_1)\cup Y_2\cup (-Y_2),$$
where
$$Y_1=\{(x,y)\in \R_+\times\R_+ \},\quad Y_2=\{(x,y)\in \R_-\times\R_+
\}.
$$
$(iii)$ The restriction of the metaplectic representation to
$H$ is given by:
\begin{equation}\label{mualfabetao}
\mu(h((s,t),(x,y)))f(u,v)=\pm e^{- s/2}e^{\pi i x u^2}e^{-
t/2}e^{\pi i y v^2}  \,f(e^{-s}u,e^{-t}v).
\end{equation}

Theorem \ref{RepThm} rephrased  for the TDW group is as follows.
\begin{theorem} The subgroup TDW is reproducing on $L^2(Y_1\cup(-Y_1))$. A function $\psi\in L^2(Y_1\cup(-Y_1))$
is a reproducing function if and only if
\begin{equation}\label{condtensor}
\int_{Y_1} |\psi(u,v)|^2 \,\frac{du\,dv}{u^2v^2}\,=\int_{Y_1}
|\psi(-u,-v)|^2 \,\frac{du\,dv}{u^2v^2}\,=1,\, \int_{Y_1}
\psi(-u,-v)\overline{\psi(u,v)} \,\frac{du\,dv}{u^2v^2}\,= 0,
\end{equation}
and, similarly, TDW is reproducing on  $L^2(Y_2\cup(-Y_2))$.
\end{theorem}

Now, recall the reproducing subgroup $H_1\subset\R^2\rtimes SL(2,\R)$ given by
$$
H_1=\left\{\left(\bmatrix 0\\
0\endbmatrix, \bmatrix 1&0\\b&1\endbmatrix\bmatrix a^{-1/2}&0\\0 &
a^{1/2}\endbmatrix\right),\,a>0,b\in\R\right\}.
$$
A reproducing function $\f_1$ is reproducing for $H_1$ if and only if
$\f_1\in L^2(\R)$ and
\begin{equation*}
\int_0^\infty |\f_1(x)|^2\frac{dx}{x^2}=\int_0^\infty
|\f_1(-x)|^2\frac{dx}{x^2}=\frac12,\quad \int_0^\infty
\f_1(x)\overline{\f_1(-x)}\frac{dx}{x^2}=0. \end{equation*}
\noindent
It is then  clear that every function $\psi(u,v)=4\f_1(u)\f_1(v)$ fulfils \eqref{condtensor}, so that a reproducing function is obtained by a tensor product of two $1$-dimensional wavelets.

\section*{Acknowledgements}
The authors would like to thank Fabio Nicola and Filippo De Mari for their
 helpful comments.



\begin{thebibliography}{10}
\bibitem{ALI2000}
{S.~T.~Ali,  J.~P. Antoine, J.~P. Gazeau},
\newblock {\em Coherent states, wavelets and their generalizations},
\newblock Springer-Verlag, New York, 2000.

\bibitem{Arnold} V.~I.~Arnold, {\em Mathematical Methods of Classical
Mechanics}, Springer-Verlag, New York, 1978.
\bibitem{BernierTaylor96}
D.~Bernier and K.~F. Taylor.
\newblock Wavelets from square-integrable representations.
\newblock {\em SIAM J. Math. Anal.}, 27(2):594--608, 1996.
\bibitem{Candes2001}
E.~J. Cand{\`e}s,  D.~L. Donoho and L.~David.
\newblock Curvelets and curvilinear integrals.
\newblock {\em J. Approx. Theory}, 113(1):59--90, 2001.
\bibitem{AEFK1}
E.~Cordero, F.~De~Mari, K.~Nowak, and A.~Tabacco.
\newblock Analytic features of reproducing groups for the metaplectic
  representation.
\newblock {\em JFAA}, 12(3):157--180, 2006.
\bibitem{AEFK2}
E.~Cordero, F.~De~Mari, K.~Nowak, and A.~Tabacco.
\newblock A dimensional bound for reproducing subgroups  of the symplectic group.
\newblock {\em Math. Nachr.}, 283(7):1--12, 2010.

\bibitem{Dalke}
S.~Dalke, G. Kutyniok, G. Steidl, and G. Teschke.
\newblock Shearlet coorbit spaces and associated Banach frames.
\newblock {\em Appl. Comput. Harmon. Anal.}, 27(2):195--214, 2009.

\bibitem{Filippo2001}
F.~De~Mari and K.~Nowak.
\newblock Analysis of the affine transformations of the time-frequency
plane.
\newblock {\em Bull. Austral. Math. Soc.}, 63(2):195--218, 2001.
\bibitem{DIX}
J.~Dixmier.
\newblock {\em Les $C^*$-Alg\`ebres et leurs repr\'esentations},
\newblock Gauthier-Villars \'Editeur, Paris, 1969.


\bibitem{FHN08}
H.G. Feichtinger, M. Hazewinkel, N. Kaiblinger, E. Matusiak and M. Neuhauser.
\newblock Metaplectic operators on {$\mathbb C^n$}.
\newblock {\em Q. J. Math.}, 59(1):15--28, 2008.

\bibitem{Fol2} G.~B.~Folland, {\em A  Course in Abstract Harmonic
Analysis}, CRC Press, Boca Raton, Florida, 1995.
\bibitem{folland89}
G.~B. Folland.
\newblock {\em Harmonic analysis in phase space}.
\newblock Princeton Univ. Press, Princeton, NJ, 1989.


\bibitem{fulton2}
W. Fulton, {\em Introduction to
intersection theory in algebraic
geometry}. CBMS Regional Conference
Series in Mathematics, 54. Published
for the Conference Board of the
Mathematical Sciences, Washington, DC;
by the American Mathematical Society,
Providence, RI, 1984.

\bibitem{Fulton} W.~Fulton, {\em Algebraic Topology},  Springer-Verlag, New York 1995.
\bibitem{book}
K.~Gr{\"o}chenig.
\newblock {\em Foundations of Time-Frequency Analysis}.
\newblock Birkh\"auser, Boston, 2001.

\bibitem{guo-labate}
K.~Guo and D. Labate.
\newblock{Sparse shearlet representation of {F}ourier integral  operators}
\newblock{\em Electron. Res. Announc. Math. Sci.},
\newblock 14:7--19, 2007.

\bibitem{Hel} S.~Helgason, {\em Differential Geometry, Lie
Groups and Symmetric Spaces}, Academic Press, New York 1978.

\bibitem{GaLi80}
C.~B. Garcia and T.~Y. Li.
\newblock On the number of solutions to polynomial systems of
equations.
\newblock {\em SIAM J. Numer. Anal.}, 17(4):540--546, 1980.

\bibitem{Hewitt-Edwin-Ross} E.~Hewitt and K.~A. Ross.
\newblock {\em Abstract harmonic analysis. {V}ol. {I}}, volume 115 of
{\em
    Grundlehren der Mathematischen Wissenschaften [Fundamental Principles
of
    Mathematical Sciences]}.
\newblock Springer-Verlag, Berlin, second edition, 1979.
\newblock Structure of topological groups, integration theory, group
    representations.


\bibitem{bookGuido}
E.~Hern\'andez, G.~L.~Weiss.
\newblock{\em A First Course on Wavelets},
\newblock CRC Press, Boca Raton, 1996.


\bibitem{Jac} N.~Jacobson, {\em Lie algebras},
Dover publications, New York, 1979.

\bibitem{Kn1} A.~W.~Knapp, {\em Lie Groups Beyond an Introduction},
Secon edition, Birkh\" auser, Boston, 2002.

\bibitem{LWWW02}
R.~S. Laugesen, N.~Weaver, G.~L. Weiss, and E.~N. Wilson.
\newblock A characterization of the higher dimensional groups associated with
  continuous wavelets.
\newblock {\em J. Geom. Anal.}, 12(1):89--102, 2002.
\bibitem{SchulzTaylor99}
\newblock Extensions of the {H}eisenberg group and wavelet analysis in
              the plane.
\newblock{\em Spline functions and the theory of wavelets (Montreal, PQ,
              1996)},CRM Proc. Lecture Notes, Amer. Math. Soc., 18:217--225,  1999.

\bibitem{WW01}
 G. Weiss, and E. N. Wilson.
\newblock The mathematical theory of wavelets.
\newblock {\em Twentieth century harmonic analysis---a celebration (Il
              Ciocco, 2000)}, NATO Sci. Ser. II Math. Phys. Chem., 33:329--366, 2001.

\bibitem{William}
J.~Williamson,
\newblock On an algebraic problem, concerning the normal forms of linear
dynamical systems,
\newblock {\em Amer. J. of Math.} 58 (1): 141--163, 1936.
\end{thebibliography}
\end{document}